\documentclass{techrep_tight}
\usepackage{latexsym,amssymb,amsmath,cleveref,amsthm,graphicx}

\usepackage[mathscr]{euscript}
\usepackage{parskip}    
\usepackage{xcolor}
\usepackage{algorithm,algorithmic}
\usepackage{enumitem}
\usepackage{mathtools}

\DeclareMathAlphabet{\mathpzc}{OT1}{pzc}{m}{it}
\DeclarePairedDelimiter{\ceil}{\lceil}{\rceil}

\usepackage{sidenotes}
\usepackage[a4paper,margin=1in, right = 2.2in,marginparwidth=1.55in]{geometry}

\usepackage{caption}
\SHORTTITLE{Submodular optimization}
\TITLE{A brief note on submodular optimization}
\AUTHORS{Alen~Alexanderian\footnote{
North Carolina State University, Raleigh, NC, USA. E-mail~\texttt{alexanderian@ncsu.edu}
        \newline{Last revised: \today}}}

\ABSTRACT{We consider some basic concepts concerning the optimization of monotone submodular
functions, including a detailed discussion of submodularity. 
Illustrative examples are included to provide further
insight.  We then discuss the greedy method and a couple of its more efficient
variants for approximately maximizing monotone submodular functions.}  

\DeclareMathOperator*{\argmax}{arg\,max}

\newcommand{\R}{\mathbb{R}}

\newcommand{\mat}[1]{\mathbf{{#1}}}
\renewcommand{\vec}[1]{{\mathchoice
                     {\mbox{\boldmath$\displaystyle{#1}$}}
                     {\mbox{\boldmath$\textstyle{#1}$}}
                     {\mbox{\boldmath$\scriptstyle{#1}$}}
                     {\mbox{\boldmath$\scriptscriptstyle{#1}$}}}}

\newcommand{\V}{\mathcal{V}}
\newcommand{\pow}{{\large\text{$\mathpzc{P}$}}}
\newcommand{\deriv}[3]{\Delta_{{#1}}({#2}\,|\,{#3})}
\newcommand{\eps}{\varepsilon}

\newcommand{\dparpr}{\vec{m}_\text{pr}}
\newcommand{\prcov}{\mat{\Gamma}_\text{pr}}
\newcommand{\ncov}{\mat{\Sigma}}

\makeatletter
\newcommand*{\defeq}{\mathrel{\rlap{%
                     \raisebox{0.3ex}{$\m@th\cdot$}}%
                     \raisebox{-0.3ex}{$\m@th\cdot$}}%
                     =}
\makeatother
\newcommand{\Vcard}{n}

\newcommand{\mysidenote}[1]{\sidenote{\small\it{\color{blue}{#1}}}}

\newcommand{\derivtwo}[4]{\Delta_{#3}\Delta_{#2} #1(#4)}

\begin{document}

\maketitle

\section{Introduction}
Let $V$ be a finite set with $\Vcard$ elements.\mysidenote{In the present 
context, $V$ is typically referred to as the ground set.}
For $k \in \{1, \ldots, \Vcard\}$, 
we let  
\begin{equation}\label{equ:Vk}
\V_k \defeq \{ A \in \pow(V) : |A| = k \},
\end{equation}
denote the collection of subsets of $V$ that have $k$ 
elements.\mysidenote{Here, $\pow(V)$ denotes the power set of $V$. 
Also, for $A \subseteq V$, $|A|$ denotes its cardinality.}
In this brief note, we consider 
optimization problems of the form 
\begin{equation}\label{equ:optim_basic}
\max_{S \in \V_k} f(S),
\end{equation}
where $f:\pow(V)\to \R$ is a non-negative monotone submodular function 
with the property that $f(\emptyset) = 0$.
Solving such problems by an exhaustive search is extremely challenging. 
This would require $\binom{\Vcard}{k}$ evaluations of $f$, which 
is prohibitive even for modest values of $\Vcard$ and 
$k$.\mysidenote{For example, we note $\binom{80}{20} = \mathcal{O}(10^{18})$.}
In this note, we discuss approximate solution of such problems using 
the greedy method and some of its variants.

We start our discussion in Section~\ref{sec:background}, where we outline the
requisite background concepts and notation.  In that section, we view the notion
of the marginal gain as a discrete derivative and discuss a discrete analogue of
the Fundamental Theorem of Calculus.  

In Section~\ref{sec:submodular_functions}, we 
discuss submodular functions, a couple of equivalent characterizations of 
submodularity, and two basic results on operations that preserve submodularity.
Subsequently, in Section~\ref{sec:submodular_examples}, we
provide further insight into submodularity by considering a few examples of
submodular functions. 

The greedy method is studied in Section~\ref{sec:greedy}. We also discuss
a well-known theoretical guarantee for the greedy algorithm in the case of
monotone submodular functions. Finally, in Section~\ref{sec:variants}, we
consider two accelerated variants of the greedy algorithm.

\section{Preliminaries}
\label{sec:background}
Let $V$ be a finite set with $|V| = n$.
Consider a set function 
$f:\pow(V) \to \R$. We will always assume $f(\emptyset) = 0$. 
We first define the notion of monotonicity. 
\begin{definition}
The function $f$ is said to be \emph{monotone}, 
if for every 
$A$ and $B$ in $\pow(V)$ such that $A \subseteq B$, we have 
$f(A) \leq f(B)$. 
\end{definition}
We next consider 
the notion of the \emph{marginal gain}, which is also 
known as the \emph{discrete derivative}~\cite{KrauseGolovin14}. 
For $A \subseteq V$ and $v \in V$, the marginal 
gain of $f$ at $A$ with respect to $v$ is defined as 
\[
\deriv{f}{v}{A} \defeq f(A \cup \{ v\}) - f(A).
\]
It is straightforward to show that 
$f:\pow(V) \to \R$ is monotone if and only if $\deriv{f}{v}{A} \geq 0$ for every $A
\subseteq V$ and $v \in V$.

The following result, which 
may be viewed as
a discrete analogue of the Fundamental Theorem of Calculus,\mysidenote{
Let $f\in C^1([a,a+h])$, where $a\in\R$ and $h>0$. By the Fundamental
Theorem of Calculus,
\[
f(a+h)
=
f(a)+\int_a^{a+h}f'(t)\,dt.
\]
}
will be useful in the discussions that follow.
\begin{lemma}
\label{lem:discrete-ftc}
Let $A,H\in\pow(V)$ and suppose $H=\{h_1,\ldots,h_k\}$ with $k \leq |V|$.
Then,
\begin{equation}\label{equ:ftc}
f(A\cup H) = f(A) +
\sum_{j=1}^k
\deriv{f}{h_j}{A\cup\{h_1,\ldots,h_{j-1}\}}.
\end{equation}
\end{lemma}
\begin{proof}
Let $A_0 = A$ and 
$A_j= A \cup \{h_1,\ldots,h_j\}$, for $j \in \{1, \ldots, k\}$.
We have 
\[
f(A_j)-f(A_{j-1})
=
\deriv{f}{h_j}{A_{j-1}},
\qquad j\in\{1,\ldots,k\}.
\]
Summing these identities over $j$ yields
$f(A_k)-f(A_0) = \sum_{j=1}^k \deriv{f}{h_j}{A_{j-1}}$.
The result follows by noting that 
$A_0=A$ and $A_k=A\cup H$.
\end{proof}

\section{Submodular functions}\label{sec:submodular_functions}
We begin by stating the definition of a \emph{submodular function}.
\begin{definition}\label{def:submodular} 
Consider a set function $f:\pow(V) \to \R$. We say $f$ is submodular, if for every 
$A \subseteq V$ and 
$B \subseteq V$ such that $A \subseteq B$,
\begin{equation}\label{equ:submodular}
    \deriv{f}{v}{A} \geq \deriv{f}{v}{B}, \quad \text{whenever } v \in V \setminus B.
\end{equation}
\end{definition} 
This definition has an intuitive interpretation---this 
is a diminishing return property.\mysidenote{The following 
figure provides an illustration of submodularity.
\includegraphics[width=0.35\textwidth]{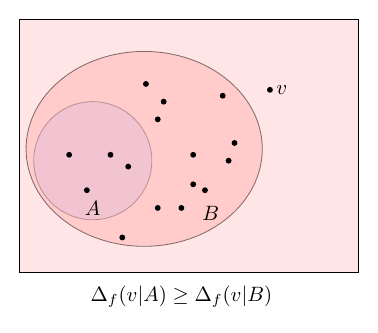}}
To make matters concrete, 
suppose the elements of $V$ correspond to a set of experiments 
and suppose $f:\pow(V)\to\R$ assigns a utility to each 
subset of $V$. Then, the above definition states that 
if the experiments in $A$ are conducted, the marginal utility 
of performing the experiment $v \in V \setminus B$ does not increase if we 
also perform the experiments in $B \setminus A$.

The following result provides an \emph{incremental} characterization 
of submodularity, which is sometimes convenient to use. 
This result is well-known; see, e.g.,~\cite{Schrijver2003}. 
The proof is straightforward and is provided for completeness. 
\begin{proposition}\label{prp:incremental-diminishing-returns}
Let $V$ be a finite set and let $f:\pow(V)\to\R$. 
The following statements are equivalent:
\begin{enumerate}[label=(\roman*)]
\item The function $f$ is submodular.
\item For every $S\subseteq V$ and every pair of distinct elements
$v,w\in V\setminus S$,
\[
\deriv{f}{v}{S}
\geq
\deriv{f}{v}{S\cup\{w\}}.
\]
\end{enumerate}
\end{proposition}
\vspace{-0.75\baselineskip}
\begin{proof}
The implication $(i)\Rightarrow(ii)$ follows by taking
$A=S$ and $B=S\cup\{w\}$ in~\eqref{equ:submodular}.
Conversely, suppose (ii) holds, and let
$A\subseteq B\subseteq V$ and $v\in V\setminus B$. 
If $A=B$, then~\eqref{equ:submodular} holds trivially. Thus, assume
$A$ is a proper subset of $B$. Assume that $B \setminus A$ has $k \geq 1$ 
elements,
$B\setminus A=\{b_1,\ldots,b_k\}$.
Define
$A_0=A$ and $A_j=A\cup\{b_1,\ldots,b_j\}$ for $j \in 
\{1,\ldots,k\}$.
Note that for each $j\geq 1$, the elements $v$ and $b_j$ are distinct and belong to
$V\setminus A_{j-1}$. Therefore,
\[
\deriv{f}{v}{A_{j-1}} \geq \deriv{f}{v}{A_{j-1}\cup\{b_j\}}
= \deriv{f}{v}{A_j}, \quad j \in \{1, \ldots, k\}.
\]
Applying this inequality successively yields,
\[
\deriv{f}{v}{A} = \deriv{f}{v}{A_0} \geq \deriv{f}{v}{A_1} \geq
\cdots \geq \deriv{f}{v}{A_k} = \deriv{f}{v}{B}.
\]
Thus, $f$ is submodular.
\end{proof}

\begin{remark}
Proposition~\ref{prp:incremental-diminishing-returns} admits an interesting 
interpretation in terms of second order discrete derivatives.
To make matters concrete, 
for $S\subseteq V$ and $v,w\in V$, we define the
mixed discrete derivative
\[
\derivtwo{f}{v}{w}{S} \coloneqq \deriv{f}{v}{S\cup\{w\}} - \deriv{f}{v}{S}.
\]
Equivalently,
$\derivtwo{f}{v}{w}{S} = f(S\cup\{v,w\}) - f(S\cup\{v\}) - f(S\cup\{w\}) + f(S)$.
Note that $\{\derivtwo{f}{v}{w}{S}\}_{v, w \in V}$ may be viewed as the entries of a \emph{discrete Hessian} matrix.
Subsequently, Proposition~\ref{prp:incremental-diminishing-returns}
states that $f$ is submodular if and only if
$\derivtwo{f}{v}{w}{S}\leq 0$
for every $S\subseteq V$ and all distinct $v,w\in V\setminus S$.
\end{remark}

Another well-known equivalent characterization of submodularity 
is provided by the following result~\cite{KrauseGolovin14};
its proof demonstrates the utility of Lemma~\ref{lem:discrete-ftc}.
\begin{proposition}\label{prp:submodular-lattice}
Let $V$ be a finite set and consider a function $f:\pow(V)\to\R$. 
Then, $f$ is submodular if and only if
\begin{equation}\label{equ:submodular_alt}
f(A\cap B)+f(A\cup B) \leq f(A)+f(B), \qquad \text{for all } A,B\in\pow(V).
\end{equation}
\end{proposition}
\vspace{-0.75\baselineskip}
\begin{proof}
Suppose first that $f$ is submodular in the sense of Definition~\ref{def:submodular}.
Let $A$ and $B$ be subsets of $V$.
If $A\subseteq B$, then
$A\cap B=A$ and $A\cup B=B$, in which case~\eqref{equ:submodular_alt} holds with equality.
Thus, we assume $A\setminus B\neq\emptyset$ and write
$A\setminus B=\{a_1,\ldots,a_k\}$.\mysidenote{
Keeping the following diagram in mind will be helpful in what follows. 
\includegraphics[width=0.35\textwidth]{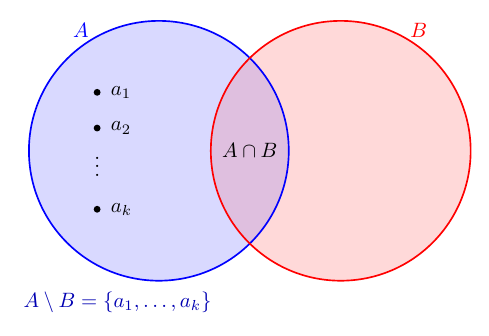}}  
Using Lemma~\ref{lem:discrete-ftc}, we have 
\[
f(A\cup B) = f(B) + \sum_{j=1}^k \deriv{f}{a_j}{B\cup\{a_1,\ldots,a_{j-1}\}}.
\]
Using Lemma~\ref{lem:discrete-ftc} again, and noting that
$A=(A\cap B)\cup(A\setminus B)$,
we have
\[
f(A) = f(A\cap B) + \sum_{j=1}^k \deriv{f}{a_j}{(A\cap B)\cup\{a_1,\ldots,a_{j-1}\}}.
\]

Note that, for each $j$, we have
$(A\cap B)\cup\{a_1,\ldots,a_{j-1}\} \subseteq B\cup\{a_1,\ldots,a_{j-1}\}$. 
Further, $a_j$ does not belong to the larger set. 
Hence, submodularity implies
\[
\deriv{f}{a_j}{(A\cap B)\cup\{a_1,\ldots,a_{j-1}\}} 
  \geq \deriv{f}{a_j}{B\cup\{a_1,\ldots,a_{j-1}\}}.
\]
Therefore, 
\[
\begin{aligned}
f(A)-f(A\cap B)
&=      \sum_{j=1}^k \deriv{f}{a_j}{(A\cap B)\cup\{a_1,\ldots,a_{j-1}\}}\\ 
&\geq   \sum_{j=1}^k  \deriv{f}{a_j}{B\cup\{a_1,\ldots,a_{j-1}\}}  
= f(A\cup B)-f(B).
\end{aligned}
\]
That is, $f(A)-f(A\cap B) \geq f(A\cup B)-f(B)$, which  
implies~\eqref{equ:submodular_alt}.

Conversely, suppose~\eqref{equ:submodular_alt} holds. Let
$S\subseteq T\subseteq V$ and let $v\in V\setminus T$. Apply
\eqref{equ:submodular_alt} with $A=S\cup\{v\}$ and $B=T$.
Note that $A\cap B=S$ and $A\cup B=T\cup\{v\}$. Thus,
\[
f(S)+f(T\cup\{v\}) \leq f(S\cup\{v\})+f(T).
\]
Therefore, $\deriv{f}{v}{S} \geq \deriv{f}{v}{T}$. Hence, $f$ is submodular.
\end{proof}

We end this section by listing a couple of useful results
regarding submodular functions. 
\begin{proposition}\label{prp:summation}
Let $V$ be a finite set, and consider 
submodular functions $f_i:\pow(V)\to\R$.
Further, let  
$\alpha_1,\ldots,\alpha_m$ be non-negative constants. Then, the function
$f\defeq\sum_{i=1}^m \alpha_i f_i$
is submodular.
\end{proposition}
\begin{proof}
Let $A\subseteq B\subseteq V$ and let $v\in V\setminus B$. Then,
\[
\deriv{f}{v}{A}
= \sum_{i=1}^m \alpha_i\deriv{f_i}{v}{A} \geq
\sum_{i=1}^m \alpha_i\deriv{f_i}{v}{B} = \deriv{f}{v}{B},
\]
where we used the submodularity of each $f_i$ and the fact that
$\alpha_i\geq 0$.
\end{proof}
The next result is useful in the 
context of stochastic submodular optimization. 
\begin{proposition}\label{prp:expectation}
Let $(\Omega,\mathcal F,\mathbb P)$ be a probability space. 
Consider a function $f:\Omega \times \pow(V) \to \R$.
Suppose
$f(\omega, \cdot):\pow(V)\to\R$ is submodular for every $\omega\in\Omega$ and 
that 
$f(\cdot, A)$ is integrable for every $A\subseteq V$. Define
the function $F:\pow(V) \to \R$ as 
\[
F(S)\defeq\mathbb E[f(\cdot, S)], \quad S \in \pow(V).
\]
Then, $F$ is submodular.
\end{proposition}
\begin{proof}
Note that for each $A \subseteq V$, 
\[
\begin{aligned}
\deriv{F}{v}{A} &= F(A\cup\{v\})-F(A) \\
&= \mathbb E\bigl[f(\cdot,A\cup\{v\})-f(\cdot,A)\bigr]\\
&=
\int_\Omega\deriv{f(\omega,\cdot)}{v}{A}\,\mathbb P(d\omega).
\end{aligned}
\]
Let $A\subseteq B\subseteq V$ and let $v\in V\setminus B$. 
Since
$f(\omega, \cdot)$ is submodular for every $\omega\in\Omega$,
\[
\deriv{f(\omega, \cdot)}{v}{A}
\geq
\deriv{f(\omega, \cdot)}{v}{B}, 
\quad \text{for all } \omega \in \Omega. 
\]
Therefore, 
\[
\deriv{F}{v}{A}=
\int_\Omega \deriv{f(\omega, \cdot)}{v}{A} \, \mathbb{P}(d\omega)\geq
\int_\Omega \deriv{f(\omega, \cdot)}{v}{B} \, \mathbb{P}(d\omega)
=
\deriv{F}{v}{B}.
\]
Thus, $F$ is submodular.
\end{proof}
The Propositions~\ref{prp:summation} and~\ref{prp:expectation} 
concern two basic operations that preserve
submodularity. For other such results, 
see, e.g.,~\cite{Fujishige2005,KrauseGolovin14}.

\begin{remark}
In Proposition~\ref{prp:expectation}, 
the assumption that $f(\omega,\cdot)$ is submodular for every
$\omega\in\Omega$ may be weakened by requiring that $f(\omega, \cdot)$ 
is submodular for \emph{almost all} $\omega \in \Omega$. 
\end{remark}

\section{Examples of submodular functions}\label{sec:submodular_examples}
In this section, we consider a couple of examples of submodular functions.

\subsection{Maximum-value function}
Let $V=\{1,\ldots,n\}$ and let $\vec{x}\in\R^n$ be a fixed vector with non-negative entries
that represent a given collection of values. 
Define $f:\pow(V)\to\mathbb{R}$ by
\begin{equation}\label{equ:maxval}
   f(S)=
   \begin{cases}
   \displaystyle \max_{j\in S} x_j, & S\neq\emptyset,\\
   0, & S=\emptyset.
   \end{cases}
\end{equation}
We refer to $f$ as the maximum-value function. 
\begin{lemma}\label{lem:maxval}
The maximum-value function is monotone and submodular.
\end{lemma}
\vspace{-0.75\baselineskip}
\begin{proof}
Monotonicity is immediate. Namely, if $\emptyset \neq A\subseteq B$, then the maximum over $A$ cannot
exceed the maximum over $B$. 
If $A=\emptyset$, $f(A) = 0 \leq f(B)$ for every $B \subseteq V$. 
To prove submodularity, let $A\subseteq B\subseteq V$ and $s\in V\setminus B$.
We have
\[
\begin{alignedat}{2}
f(A\cup\{s\})-f(A)
&= \max\{f(A),x_s\}-f(A) \\
&= \max\{0,x_s-f(A)\}\\ 
&\ge \max\{0,x_s-f(B)\} 
&\quad& \text{({\color{blue}Since $f(A)\le f(B)$})} \\
&= f(B\cup\{s\})-f(B). \qedhere
\end{alignedat}
\]
\end{proof}
This example provides a simple illustration of the diminishing returns property
encoded in submodularity.  Specifically, in the case of the maximum-value
function, adding a new element $s$ to a set $A$ only improves the maximum if
$x_s$ exceeds the current maximum value in $\{x_j\}_{j\in A}$.  However, the
current maximum value can only increase when passing from a set $A$ to a superset
$B$. Therefore, the marginal gain can only decrease.

\subsection{The facility function}
Let $V = \{1, \ldots, n\}$ and consider a matrix $\mat{M} \in \R^{m \times n}$ with 
non-negative entries.
Define the function $f:\pow(V) \to \R$ by 
\begin{equation}\label{equ:facility}
f(S) = \left\{
      \begin{alignedat}{2} 
          &\displaystyle\sum_{i=1}^m \displaystyle\big(\max_{j \in S} M_{ij} \big) 
             \quad &\text{if } &S \neq \emptyset, \\
          &0 
             \quad &\text{if } &S = \emptyset.
      \end{alignedat}
\right.
\end{equation}
This function, which is commonly referred to as the facility function, is associated with the
scenario where one seeks to open facilities in $n$ candidate locations to serve
$m$ customers. Here, $M_{ij}$ is the value provided by the
facility at the $j$th location to the $i$th customer. Assuming each customer
chooses the facility providing the highest value to them, $f$
quantifies the total value provided by a given configuration of facilities; 
see Figure~\ref{fig:facility} for an illustration.
\begin{figure}[ht]\centering
\includegraphics[width=.65\textwidth]{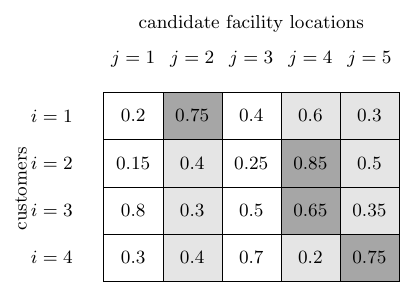}
\caption{Simple illustration of the facility function. 
Here, the facility configuration is $S = \{2, 4, 5\}$, 
corresponding to grayed columns. The dark gray cells show 
the facility selected by the $i$th customer, for $i \in \{1, 2, 3, 4\}$.
For the displayed configuration, 
$f(S)=0.75+0.85+0.65+0.75=3.00$.}
\label{fig:facility}
\end{figure}

The present facility function is a well-known example of a monotone 
submodular function~\cite{KrauseGolovin14}.  
This is detailed in the following result. 
\begin{proposition}
The function $f$ defined in~\eqref{equ:facility} is monotone and submodular.
\end{proposition}
\vspace{-0.75\baselineskip}
\begin{proof}
The monotonicity property is straightforward 
to note. To show submodularity,   
for each $i\in \{1, \ldots, m\}$, we consider $f_i:\pow(V) \to \R$ defined as 
\[
   f_i(S)=
   \begin{cases}
   \displaystyle\max_{j\in S} M_{ij}, & S\neq \emptyset,\\
   0, & S=\emptyset.
   \end{cases}
\]
Note that $f = \sum_{i=1}^m f_i$.
Furthermore, 
by Lemma~\ref{lem:maxval}, $f_i$ is 
submodular for each $i$. 
Hence, since submodularity is preserved under finite sums (cf.\ Proposition~\ref{prp:summation})
$f$ is submodular.
\end{proof}

\subsection{Expected information gain in linear Gaussian inverse problems}
Consider the inverse problem of estimating $\vec m \in \R^n$ from the
data model
\[
\vec{y} = \mat{F}\vec{m} + \vec{\eta},
\]
where $\vec y \in \R^d$ is a vector of sensor data,
$\mat{F}\in\R^{d\times n}$ is the forward operator, and
$\vec{\eta} \sim N(\vec{0}, \ncov)$ represents measurement noise. We assume a
Gaussian prior $N(\dparpr,\prcov)$ for $\vec m$, and assume that $\vec m$ and
$\vec\eta$ are independent. In the present setup, each row of $\mat{F}$
corresponds to a candidate sensor location.
In practice, due to budget constraints, sensors can be placed only at a subset
of the candidate locations. It is therefore important to identify an optimal
placement of sensors. 

Let $V = \{1,\ldots,d\}$ index the set of candidate sensor
locations. The optimal sensor placement problem seeks to find a subset
$S \subseteq V$ that provides maximal information gain.
For $S \subseteq V$, let $\mat{F}_S$ denote the restricted forward
operator obtained by retaining the rows of $\mat{F}$ with indices in $S$, and
let $\ncov_S$ denote the noise covariance matrix obtained by restricting
$\ncov$ to the rows and columns with indices in $S$.
A common objective in Bayesian optimal experimental design is to select $S
\subseteq V$ to maximize the expected information gain (EIG), defined as the
expected Kullback--Leibler divergence from the posterior distribution to the
prior distribution. For linear Gaussian inverse problems, the EIG admits the
closed-form expression
\[
\Psi(S)
=
\frac12
\log\det\!\big(
\mat{I}+\prcov^{1/2} \mat{F}_S^\top \ncov_S^{-1}\mat{F}_S\prcov^{1/2}
\big).
\]
Under the assumption of uncorrelated measurement errors, i.e., when $\ncov$ is
diagonal, $\Psi:\pow(V) \to \R$ is known to be a monotone submodular function;
see, e.g.,~\cite{AlexanderianMaio26}.  On the other hand, when the measurement
errors are correlated, submodularity can fail. This is illustrated next.

Consider the case where
\begin{equation}\label{equ:FSigma}
\mat{F} =
\begin{bmatrix}
1 & 1 & 1\\
2 & 1 & 0\\
0 & 1 & 2
\end{bmatrix}
\quad\text{and}\quad
\ncov =
\begin{bmatrix}
1 & 0 & 0\\
0 & 1 & \rho\\
0 & \rho & 1
\end{bmatrix},
\end{equation}
with $\rho \in (-1, 1)$ fixed. Also, we let the prior distribution be
$N(\vec{0}, \mat{I})$. In this case,
\[
\Psi(S)
=
\frac12\log\det\!\left(\mat{I} + \mat{F}_{\!S}^\top \ncov_{\!S}^{-1}
\mat{F}_{\!S} \right),
\qquad S \subseteq \{1, 2, 3\}.
\]

For $\rho = 1/2$, a direct calculation yields
\[
\begin{aligned}
\Psi(\{1,2\})-\Psi(\{1\})
&=
\frac12\log(15/4)
\approx 0.6609,
\\
\Psi(\{1,2,3\})-\Psi(\{1,3\})
&=
\frac12\log(24/5)
\approx 0.7843.
\end{aligned}
\]
That is,
$\deriv{\Psi}{2}{\{1\}}
<
\deriv{\Psi}{2}{\{1,3\}}$,
which violates submodularity. Here, the measurement errors of the second and
third sensors are correlated, and activating the third sensor enhances the
marginal gain of activating the second one.

For further insight, In Figure~\ref{fig:eigfig}, we plot 
$\deriv{\Psi}{2}{\{1\}}$ and $\deriv{\Psi}{2}{\{1,3\}}$ for 
$\rho \in (-0.9,0.9)$ in~\eqref{equ:FSigma}.
We observe that submodularity is violated once the magnitude of the correlation
between the errors of sensors 2 and 3 is sufficiently large.
\begin{figure}[ht]\centering
\includegraphics[width=.5\textwidth]{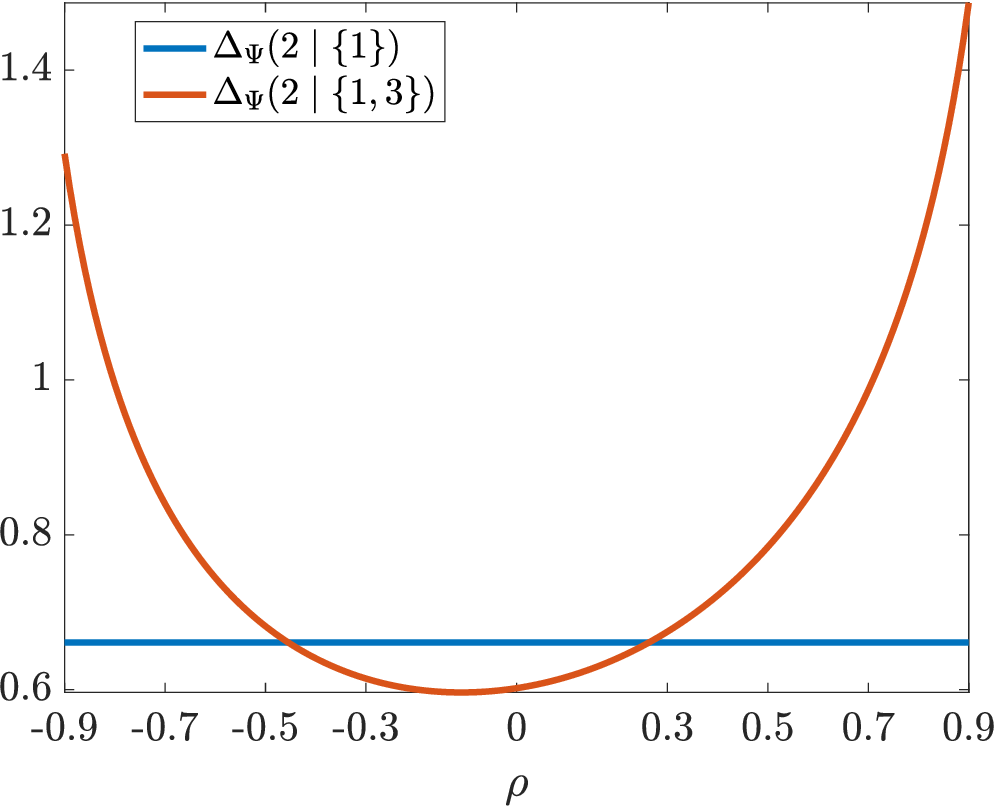}
\caption{The marginal gains $\deriv{\Psi}{2}{\{1\}}$ and $\deriv{\Psi}{2}{\{1,3\}}$ as the
correlation between measurement errors of the second and third sensors varies.}
\label{fig:eigfig}
\end{figure}

\section{The greedy method}
\label{sec:greedy}
Consider the optimization problem~\eqref{equ:optim_basic} and
assume $f$ is a monotone submodular function.  Note that, in general, one
considers optimization of $f$ over subsets of $V$ that have less than or equal
to $k$ elements. However, since $f$ is assumed monotone,  we consider the setup
in~\eqref{equ:optim_basic} to keep the discussion simple. 
A straightforward approach to approximately solving \eqref{equ:optim_basic} is the greedy method. 
In this approach, 
we begin with the empty set $S = \emptyset$, and in each iteration 
pick the element in $V \setminus S$ that provides the largest marginal gain. 

Specifically, the greedy algorithm applied to~\eqref{equ:optim_basic}
produces a finite collection of sets $\{S_l\}_{l=1}^k$ as follows: 
\begin{equation}\label{equ:greedy}
\left\{ 
\begin{alignedat}{2}
&S_0 &= &~\emptyset, \\
&S_{l} &= &~S_{l-1} \cup \{ \argmax_{v \in V \setminus S_{l-1}} \deriv{f}{v}{S_{l-1}} \},
\quad l= 1, \ldots, k.
\end{alignedat}
\right.
\end{equation}
The output of the greedy algorithm is the set $S = S_k$, which provides an 
approximate solution to~\eqref{equ:optim_basic}. 

For each $l \in \{1, \ldots, k\}$, let $v_l \in \argmax_v
\deriv{f}{v}{S_{l}}$.  That is $v_{l}$ is the element of $V$ that is
selected in the step $l$ of the greedy algorithm so that $S_{l} =
S_{l-1} \cup \{v_l\}$.  
Note that there might be more than one element in 
$V \setminus S_{l-1}$ that maximize the marginal gain at that step. 
In such cases some form of tie-break rule must be used. The manner in which 
this is done does not impact the analysis that follows. Next, 
note that for every $l \in \{0, \ldots, k-1\}$
and $v \in V$,
\begin{equation}\label{equ:greedy_basic_estimate}
f(S_{l+1}) - f(S_l) = f(S_l \cup \{v_{l+1}\}) - f(S_l) 
= \deriv{f}{v_{l+1}}{S_{l}} \geq \deriv{f}{v}{S_{l}}.
\end{equation}
This observation will be revisited shortly. 

The greedy algorithm 
is popular due to 
its simplicity. This is useful for example in the context of 
sensor placement, where one can place sensors in a greedy manner. 
In practice, this approach often provides near optimal sensor placements. 
For monotone submodular functions, this approach admits 
a theoretical guaranty. This is made precise in Theorem~\ref{thm:greedy}. 
This important 
result was proven in~\cite{NemhauserWolseyFisher78}. We 
provide a proof of this result for completeness. The 
idea behind the proof belongs to~\cite{NemhauserWolseyFisher78}. 
The present proof was adapted from the presentation in~\cite{KrauseGolovin14}.
\begin{theorem}\label{thm:greedy}
Consider a finite set $V = \{v_1, \ldots, v_\Vcard\}$ 
and assume $f:\pow(V) \to [0, \infty)$ is a monotone submodular 
function  with $f(\emptyset) = 0$. Let $\{S_l\}_{l=1}^k$ be produced by the greedy 
procedure~\eqref{equ:greedy}, for a given $k \in \{1, \ldots, \Vcard\}$.
Then,
\[
f(S_k) \geq (1-1/e) \max_{S \in \V_k}f(S),
\]
where $\V_k$ is the collection of subsets of 
$V$ with $k$ elements as defined in~\eqref{equ:Vk}.
\end{theorem}
\begin{proof}
The result holds trivially for $k = 1$. Thus, we assume $k > 1$.
Let $S^* \in \argmax_{S \in \V_k} f(S)$. Enumerate elements of 
$S^*$ as $S^* = \{s_1^*, s_2^*, \ldots, s_k^*\}$. We note that 
for each $l \in \{1, \ldots, k-1\}$,
\[
\begin{alignedat}{2}
f(S^*) 
&\leq f(S^* \cup S_l) &\quad &(\text{\color{blue}by monotonicity of $f$})\notag\\
&= f(S_l) + \sum_{j=1}^k \deriv{f}{s_j^*}{S_l \cup \{s_1^*, \ldots, s_{j-1}^*\}}
  &\quad &(\text{\color{blue}cf.~Lemma~\ref{lem:discrete-ftc}})\notag\\
&\leq f(S_l) + \sum_{j=1}^k \deriv{f}{s_j^*}{S_l} 
  &\quad &(\text{\color{blue}by submodularity of $f$})\notag\\
&\leq f(S_l) + \sum_{j=1}^k [f(S_{l+1}) - f(S_l)] 
  &\quad &(\text{\color{blue}cf.~\eqref{equ:greedy_basic_estimate}})\notag\\
&= f(S_l) + k [f(S_{l+1}) - f(S_l)]. 
\end{alignedat}
\]
Hence, 
$f(S^*) - f(S_l) \leq k [f(S_{l+1}) - f(S_l)]$.
Let $\delta_l \defeq f(S^*) - f(S_l)$ and note 
\begin{equation}\label{equ:improvement}
  \delta_l \leq k\,(\delta_{l} - \delta_{l+1}).
\end{equation}
This can be restated as $\delta_{l+1} \leq (1-1/k)\delta_l$.
Note also that 
$\delta_0 = f(S^*) - f(\emptyset) = f(S^*)$. Therefore, we have
\begin{equation}\label{equ:almost_there}
  \delta_{k} \leq  (1-1/k)^k \delta_0 =  (1-1/k)^k f(S^*) 
  \leq e^{-1}f(S^*).  
\end{equation}
In the last step, we have used the fact that $1-x \leq e^{-x}$ for every 
$x \in \R$. Substituting $\delta_k = f(S^*) - f(S_k)$
in~\eqref{equ:almost_there}, yields $f(S^*) - f(S_k) \leq  e^{-1}f(S^*)$. 
That is, 
\[
f(S_k) \geq (1-1/e)f(S^*),
\]
which is the desired result.
\end{proof}

\begin{remark}
In~\eqref{equ:improvement}, $\delta_l$ 
measures the gap between $f(S_l)$ and the optimal objective value. 
An interpretation of~\eqref{equ:improvement} is that the 
improvement in optimality gap, at the step $l$ of the 
greedy algorithm, is at least $\big(f(S^*) - f(S_{l-1})\big)/k$.
\end{remark}

\section{Two variants of the greedy method}
\label{sec:variants}

The greedy approach, while simple, can become 
prohibitive 
as the cardinality of $V$ grows. This is especially the 
case in problems of optimal sensor placement where each 
evaluation of $f$ might be expensive. 
Over the years, several 
variants of the greedy approach have been proposed to 
accelerate computations. Here, we discuss two such 
approaches: the lazy greedy and the stochastic greedy.

\subsection{Lazy greedy}
The idea behind the lazy greedy algorithm~\cite{Minoux78} is to make maximum 
use of submodularity to reduce the number of function evaluations. 
Recall that at the step $l$ of the greedy method, 
we find an element $v \in V$ with maximum 
marginal gain $\deriv{f}{v}{S_{l-1}}$; see~\eqref{equ:greedy}.
Moreover, by submodularity of $f$, we have that 
\[
\deriv{f}{v}{S_{l-1}} \geq \deriv{f}{v}{S_{l}}.
\]
Thus, instead of naively computing all the requisite marginal gains 
at each step of the greedy algorithm, we can 
use the already computed marginal gains as an upper bound 
for the subsequent ones. 

Let us briefly outline the lazy greedy process. 
Starting with $S_0 = \emptyset$, 
the first iteration of the lazy greedy method is 
the same as standard greedy. We compute 
$\rho(v) \defeq \deriv{f}{v^*}{S_{0}} = f(\{v\})$ for every $v \in V$. 
Subsequently, we select $s_1 \in V$ that maximizes 
$\rho(v)$ and let $S_1 = \{s_1\}$. 
Then, we sort the values of $\{\rho(v)\}_{v \in V \setminus S_1}$ 
in descending order. This sorted set of marginal 
gains will be maintained and updated in the subsequent 
iterations of the lazy greedy method. 
At the $l$th step of this method, we 
perform the following steps:
\begin{enumerate}[label=(\roman*)]
\item take an entry $v^* \in V \setminus S_{l-1}$ 
that maximizes $\rho(v)$;
\item compute $\deriv{f}{v^*}{S_{l-1}}$ and let 
$\rho(v^*) = \deriv{f}{v^*}{S_{l-1}}$;
\item if $v^*$ still maximizes $\{\rho(v)\}_{v \in V \setminus S_{l-1}}$
then let $S_l = S_{l-1}\cup\{v^*\}$ and go to step $l+1$ of the 
lazy greedy procedure.
Otherwise, go back to (i).
\end{enumerate}
While preserving the approximation guarantee of the standard greedy, the lazy
greedy procedure often provides massive improvements over the standard
greedy~\cite{LeskovecKrauseGuestrinEtAl07}.

\subsection{Stochastic greedy}
The stochastic greedy~\cite{MirzasoleimanBadanidiyuruKarbasiEtAl15} provides
further improvements to the standard greedy procedure and its lazy counterpart. 
Recall that in step $l$ of the standard greedy method, 
we need to compute the marginal gains corresponding to 
all elements in $V \setminus S_l$, which requires $\Vcard - l$ function evaluations.
The idea of stochastic greedy is to select a random sample from this set of $\Vcard - l$ 
elements and evaluate the marginal gain for this randomly chosen sample.
For clarity, we outline the stochastic greedy procedure in Algorithm~\ref{alg:stoch_greedy}.

\begin{algorithm}[H]
  \begin{algorithmic}[1]
  \STATE \textbf{Input:} monotone submodular function $f$, ground set $V$ of size $\Vcard$, size 
  $k$ of the desired subset of $V$, and sample size $s$. 
  \STATE \textbf{Output:} A near optimal set $A$ with $|S| = k$.
  \STATE $S = \emptyset$
  \FOR{$l = 1$ \TO $k$} 
  \STATE Draw a random sample set $R$ of size $s$ from $V \setminus S$
  \STATE $v_l = \argmax_{v \in R} \deriv{f}{v}{S}$ 
  \STATE $S = S \cup \{ v_l \}$.
  \ENDFOR
  \end{algorithmic}
  \caption{Stochastic greedy algorithm.}
  \label{alg:stoch_greedy}
\end{algorithm}

Note that due to its randomized nature, the output of
Algorithm~\ref{alg:stoch_greedy} is not deterministic.  Hence, the approximation
guarantee of the method is stated in expectation.  The following result
from~\cite{MirzasoleimanBadanidiyuruKarbasiEtAl15} makes matters precise. 
\begin{theorem}\label{thm:stoch_greedy}
Let $V$ be finite set with $\Vcard$ elements and 
assume 
$f:\pow(V) \to [0, \infty)$ is a monotone submodular function with $f(\emptyset) = 0$.
Let $\eps \in (0, 1)$ be given and  
$S$ be a set of cardinality $k$ obtained by performing $k$ steps of 
the stochastic greedy algorithm with a sample size of 
$s = \ceil{\frac{\Vcard}{k}\log\frac1\eps}$. Then, 
\[
\mathbb{E}\{ f(S_k) \} \geq (1-1/e-\eps) \max_{S \in \V_k}f(S),
\] 
\end{theorem}
Note that to achieve the requisite approximation guarantee, the size $s$ of the
sample set needed in the stochastic greedy procedure is required to be $s =
\ceil{\frac{\Vcard}{k}\log\frac1\eps}$, as stated in Theorem~\ref{thm:stoch_greedy} above.\mysidenote{In this context $\Vcard$ is typically large and
the ratio $\Vcard / k$ is small.}  
Thus, the cost of performing $k$ steps of the
algorithm is $\ceil{\Vcard\log\frac1\eps}$ function evaluations. 
The stochastic greedy approach provides enormous computational savings over the
greedy approach and its lazy variant.  Furthermore, stochastic greedy often
provides solutions whose performance is close to those obtained from standard
greedy approach. Additionally, as discussed in 
\cite{MirzasoleimanBadanidiyuruKarbasiEtAl15} the stochastic greedy can be
made more efficient by incorporating lazy evaluations. 

\bibliographystyle{abbrv}
\bibliography{refs}

\end{document}